\documentclass[11pt,letterpaper]{amsart}
\usepackage{amsmath,amsfonts}
\usepackage{graphicx}
\usepackage{amssymb}
\usepackage{amsthm}
\usepackage{yfonts}
\usepackage{kpfonts}
\usepackage{tikz-cd}
\usepackage{stmaryrd}
\usepackage{thm-restate}
\usepackage[toc]{appendix}

\usepackage{thmtools}
\usepackage{thm-restate}

\usepackage{hyperref}

\usepackage{mathtools}
\mathtoolsset{showonlyrefs}

\numberwithin{equation}{section}

\usepackage[usenames,dvipsnames]{pstricks}
\usepackage{epsfig}
\usepackage{rotating}
\usepackage{pst-plot}
\usepackage{pst-eps}
\usepackage{pst-grad}
\usepackage{pstricks-add}
\usepackage{lmodern}
\usepackage{xcolor}
\usepackage{graphicx}
\usepackage[top=3cm,bottom=3cm,left=3cm,right=3cm,a4paper]{geometry}

\usepackage{cite}

\newcommand{\qq}{{\mathbb Q}}

\newcommand{\cc}{{\mathbb C}}

\newcommand{\pp}{{\mathbb P}}

\newcommand{\Kod}{{\rm{Kod}}}

\newcommand{\mbar}{\overline{\mathcal{M}}}
\newcommand{\M}[2]{\mathcal{M}_{{#1}, {#2}}}
\newcommand{\Mbar}[2]{\overline{\mathcal{M}}_{{#1}, {#2}}}
\newcommand{\Mbarstack}[2]{\overline{\rm{M}}_{{#1}, {#2}}}

\newcommand{\BD}{\text{boundary divisors}}
\newcommand{\coho}[2]{\ensuremath{\mathrm{H}^{#1}\left(#2\right)}}

\newcommand{\set}[1]{\left\{ #1 \right\}}
\newcommand{\setcondition}[2]{\left\{ #1 \,\middle|\, #2 \right\}}
\newcommand{\restr}[2]{{
  \left.\kern-\nulldelimiterspace
  #1
  \vphantom{\big|}
  \right|_{#2} 
  }}

\newcommand{\lra}{\longrightarrow}

\theoremstyle{plain}
\newtheorem{theorem}{Theorem}[section]
\newtheorem{lemma}[theorem]{Lemma}
\newtheorem{proposition}[theorem]{Proposition}
\newtheorem{corollary}[theorem]{Corollary}

\theoremstyle{definition}
\newtheorem{remark}[theorem]{Remark}

\newtheorem{example}[theorem]{Example}

\title{The Kodaira classification of the moduli space of pointed curves in genus $3$}
\author[d.P. Ruben]{Ruben de Preter}

\address[d.P. Ruben]{
Department of Mathematics, Universiteit Antwerpen\\
Middelheimlaan 1, 2025 Antwerpen, Belgium} 
\email{ruben.depreter@uantwerpen.be}

\vspace{-1cm}
\begin{document}

\maketitle
\vspace{-0.9cm}
\begin{abstract}
We complete the Kodaira classification of the moduli spaces $\overline{\mathcal{M}}_{g,n}$ of curves with marked points in genus $g=3$, by proving that $\overline{\mathcal{M}}_{3,n}$ is of general type for $n \geq 15$.
We prove that the singularities of $\overline{\mathcal{M}}_{3,n}$ impose no adjunction conditions for $n \geq 1$ and that the canonical class of $\overline{\mathcal{M}}_{3,n}$ is big for $n \geq 15$.  
\end{abstract}

\nopagebreak

\section{Introduction}
The birational geometry of moduli spaces of curves $\overline{\mathcal{M}}_{g,n}$ has a rich and fascinating history.
It was discovered by Severi in $1915$ that $\mbar_g$ is unirational for low genus $g \leq 10$ \cite{Sev}. For growing $g$ the moduli space $\mbar_g$ becomes more complex from a birational viewpoint, which became apparent in the 1980's when Eisenbud, Harris and Mumford proved that $\overline{\mathcal{M}}_g$ is of general type for $g \geq 24$, see \cite{EH,Ha,HM}.
Since then, it has been a famous open problem to compute the Kodaira dimension of $\mbar_g$ in the remaining cases, and more generally to compute the Kodaira dimension of $\Mbar{g}{n}$.
A lot of progress has been made towards a full Kodaira classification of $\Mbar{g}{n}$.
For $g$ and $n$ low $\Mbar{g}{n}$ tends to be uniruled, see \cite{AB,CF} and the references therein. In the other direction
Farkas, Jensen, and Payne proved recently that also $\mbar_{22}$ and $\mbar_{23}$ are of general type, see \cite{FJP}.
For each genus $4 \leq g \leq 21$ Logan determined in \cite{Lo} an integer $n(g)$ such that $\Mbar{g}{n}$ is of general type for $n \geq n(g)$. 
The value of $n(g)$ for which this holds has been lowered for several $g$ by Farkas in \cite{Fa1}.
For genus $g \leq 2$ the Kodaira classification of $\Mbar{g}{n}$ is complete, see \cite{BF,BMS}.\\

In genus $g=3$ much less is known. The moduli space $\Mbar{3}{n}$ is rational for $n \leq 14$ \cite{CF}, and these are the only cases in genus $3$ where the Kodaira dimension is known.
The main goal of the current paper is to complete the Kodaira classification in genus $3$:

\begin{theorem}\label{Main theorem}
    The moduli space $\overline{\mathcal{M}}_{3,n}$ is of general type for $n \geq 15$.
\end{theorem}

Theorem \ref{Main theorem} confirms in genus $3$ the expectation that for $g \geq 2$ there are only finitely many values $(g,n)$ for which $\Mbar{g}{n}$ is not of general type, see  \cite{BMS} and \cite{HM,Lo} for the cases $g=2$ and $g \geq 4$  respectively.\\

There is a standard approach, used already in \cite{HM} and \cite{Lo}, to prove that the moduli space $\Mbar{g}{n}$ is of general type. We will roughly follow this approach, and we outline the key steps below.\\

The first step is proving that the canonical divisor of $\Mbar{g}{n}$ is big, i.e. a sum of an ample and an effective class. This step is intimately related to the computation of classes of effective divisors in the Picard group.
In the current paper we will use effective divisor classes computed in \cite{HM,Fa1} to study the bigness of the canonical class ${K_{\Mbar{3}{n}}}$.\\
As a first step towards Theorem \ref{Main theorem}, we prove in Section \ref{sec:K is big} the following theorem.

\begin{theorem}\label{Main thm: K is big}
    The canonical class $K_{\overline{\mathcal{M}}_{3,n}}$ is big for $n \geq 15$.
\end{theorem}

The fact that the canonical class $K_{\Mbar{g}{n}}$ is big, is not yet sufficient to prove that $\Mbar{g}{n}$ is of general type. The problem is that $\Mbar{g}{n}$ is usually singular, whereas the Kodaira dimension is defined as the Iitaka-dimension of the canonical class of a smooth projective model.\\
The second step overcomes this obstruction by proving that the singularities of $\Mbar{g}{n}$ impose no adjunction conditions. This means that any pluri-canonical form defined on the regular locus of $\Mbar{g}{n}$ lifts to a pluri-canonical form on a resolution of singularities $\widetilde{\mathcal{M}}_{g,n} \lra \overline{\mathcal{M}}_{g,n}$.\\

For the second step, Harris and Mumford proved in \cite[Theorem 1]{HM} that the singularities of $\mbar_g$ impose no adjunction conditions for $g \geq 4$.
Logan extended this in \cite[Theorem 2.5]{Lo}, by proving that the singularities of $\Mbar{g}{n}$ impose no adjunction conditions for $g \geq 4$.
In genus $3$, a thorough analysis of the singularities is not yet available. A large part of the current paper is devoted to carrying out such analysis, and the result is:

\begin{theorem}\label{Main thm: lifting forms}
    Suppose $n \geq 1$.
    The singularities of $\Mbar{3}{n}$ impose no adjunction conditions.
    In particular there is an isomorphism
    \[
    {\rm{H}}^0\left(\overline{\mathcal{M}}_{3,n},\mathcal{O}_{\overline{\mathcal{M}}_{3,n}}\left(mK_{\overline{\mathcal{M}}_{3,n}}\right)\right) \simeq {\rm{H}}^0\left(\widetilde{\mathcal{M}}_{3,n},\mathcal{O}_{\widetilde{\mathcal{M}}_{3,n}}\left(mK_{\widetilde{\mathcal{M}}_{3,n}}\right)\right)
    \]
    and $\Kod\left(\Mbar{3}{n}\right) = \kappa\left(\Mbar{3}{n},K_{\Mbar{3}{n}}\right)$.
\end{theorem}

Theorem \ref{Main theorem} follows from Theorem \ref{Main thm: K is big} and Theorem \ref{Main thm: lifting forms}.\\

The proof of Theorem \ref{Main thm: lifting forms} is based on a detailed understanding of the singularities of $\Mbar{g}{n}$. 
Singularities of the moduli space $\Mbar{g}{n}$ occur due to non-trivial autormorphisms on the underlying pointed curves being parameterized. This results in finite quotient singularities on $\Mbar{g}{n}$.
The Reid-Shepherd-Barron-Tai criterion \ref{Reid-Tai OG} provides a way to control the locus of non-canonical singularities for varieties with finite quoteint singularities, and we will use the criterion to classify the locus of non-canonical singularities in Corollary \ref{cor:non-canonical locus}.
The classical version of the criterion gives no information about \textit{how bad} the non-canonical singularities are, if there are any.
In Subsection \ref{subsec:local description Mgn} we describe several variants of the Reid-Tai criterion, which allow us to prove that some pluri-canonical forms lift to a resolution of singularities, even around the non-canonical locus. The result is:

\begin{proposition}\label{prop:condition to lift forms}
    Let $n \geq 1$ and let $\eta$ be a global $m$-canonical form on $\Mbar{3}{n}$. Suppose that the order of vanishing of $\eta$ along the divisor $\Delta_{1,\emptyset}$ is at least $m$. Then $\eta$ lifts to a resolution of singularities of $\Mbar{3}{n}$.
\end{proposition}

In Section \ref{sec:rigid component}, we use a test curve to identify a rigid components of the canonical divisor:

\begin{proposition}\label{prop:rigid component}
    For $n \geq 1$, the pluri-canonical divisor $mK_{\Mbar{3}{n}}$ has rigid component $4m\Delta_{1,\emptyset}$. In particular, every global $m$-canonical form on $\Mbar{3}{n}$ vanishes along $\Delta_{1,\emptyset}$ to order at least $4m$.
\end{proposition}

Theorem \ref{Main thm: lifting forms} follows from Proposition \ref{prop:condition to lift forms} and Proposition \ref{prop:rigid component}.

\begin{remark}
    In the seminal paper \cite{HM}, Harris and Mumford explicitly construct a resolution of singularities in a neighborhood of the locus of non-canonical singularities of $\mbar_{g}$ for $g \geq 4$, and  use this resolution to prove that the singularities impose no adjunction condition.
    This is the main part where our argument differs from the original argument.
    We use a finer version of the Reid-Tai criterion, due to \cite{CCM}, to control the adjunction conditions imposed by the non-canonical singularities. Then we use a test curve to identify a rigid component, whereby the adjunction conditions are satisfied by all global pluri-canonical forms.
\end{remark}

\subsection*{Acknowledgements}
I thank Ignacio Barros for introducing me to this problem and I am gratefull to Scott Mullane for his helpful suggestions and comments while writing the paper.
My research has been supported by Research Foundation - Flanders (FWO) within the framework of the Odysseus program project number G0D9323N.

\section{Preliminaries}
\subsection{Divisors on $\Mbar{g}{n}$}\label{sec:divisors}
The moduli space $\overline{\mathcal{M}}_{g,n}$ has finite quotient singularities, and hence is normal and $\mathbb{Q}$-factorial. In particular, we identify $\mathbb{Q}$-line bundles and $\mathbb{Q}$-Weil divisor classes on $\overline{\mathcal{M}}_{g,n}$.
The morphism $\Mbarstack{g}{n} \lra \Mbar{g}{n}$, from the moduli stack of curves to the coarse moduli space, induces a map between Picard groups
\[
\operatorname{Pic}\left(\Mbar{g}{n}\right) \lra \operatorname{Pic}\left(\Mbarstack{g}{n}\right),
\]
which becomes an isomorphism after base change to $\mathbb{Q}$:
\[
\operatorname{Pic}_\mathbb{Q}\left(\Mbar{g}{n}\right) \simeq \operatorname{Pic}_\mathbb{Q}\left(\Mbarstack{g}{n}\right).
\]
Therefore we will identify $\mathbb{Q}$-line bundles on $\Mbar{g}{n}$ and $\Mbarstack{g}{n}$.\\
The divisor $\Delta_{irr} \subseteq \overline{\mathcal{M}}_{g,n}$ is defined as the locus of pointed curves obtained by gluing together two points on a stable $(n+2)$-pointed curve of genus $g-1$ to a node. The general point in $\Delta_{irr}$ is irreducible, hence the notation.
Let $i \in \{0,1,...,g\}$ and let $S \subseteq \set{1,...,n}$. The divisor $\Delta_{i,S} \subseteq \overline{\mathcal{M}}_{g,n}$ is defined as the locus of pointed curves with a separating node that separates the curve into two components of genus $i$ and $g-i$, such that the marked points on the genus $i$ components are exactly those indexed by $S$.
For this to be well defined we have to assume that $S$ has at least $2$ elements if $i=0$ and at most $n-2$ elements if $i=g$.
The union of the divisors $\Delta_{irr}$ and $\Delta_{i,S}$ is the boundary $\Mbar{g}{n}\backslash \M{g}{n}$.
Note that
\[
\Delta_{i,S} = \Delta_{g-i,\set{1,...,n}\backslash S}.
\]
We use the notations $\delta_{irr} = [\Delta_{irr}]$ and
\[
\delta_{i,S} = \begin{cases}
    \frac{\left[\Delta_{i,S}\right]}{2} \text{ if } (i,S) = (1,\emptyset),\\
    [\Delta_{i,S}] \text{ otherwise}
\end{cases}
\]
to denote the classes in the Picard group of the boundary divisors on the moduli stack $\Mbarstack{g}{n}$.
The reason for dividing by $2$ in the definition of $\delta_{1,\emptyset}$ is that the general curve in $\Delta_{1,\emptyset}$ has $2$ automorphisms.
We denote
\[
\delta = \delta_{irr} + \sum_{i=0}^{\lfloor \frac{g}{2} \rfloor} \delta_{i},\,\,\,\, \delta_i = \sum_{S \subseteq \set{1,...,n}} \delta_{i,S},
\]
where the latter sum is indexed by all subsets $S \subseteq \set{1,...,n}$ that satisfy $|S| \geq 2$ if $i=0$, and $n \not\in S$ if $i = \frac{g}{2}$ (this last conditions makes sure that the divisor $\delta_{\frac{g}{2},S}=\delta_{\frac{g}{2},\set{1,...,n}\backslash S}$ appears only once in the sum).

Let $\rho : \mathcal{X} \lra \Mbarstack{g}{n}$ be the universal family with markings $\sigma_1,...,\sigma_n : \Mbarstack{g}{n} \lra \mathcal{X}$, and let $\omega_\rho$ be the relative dualizing sheaf of $\rho$.
The Hodge class $\lambda$ and the $\psi$-classes $\psi_1,...,\psi_n$ are defined as
\[
\lambda = c_1(\rho_*\omega_\rho) \text{ and } \psi_j = c_1\left(\sigma_j^*(\omega_\rho)\right) \in \operatorname{Pic}_\mathbb{Q}\left(\Mbar{g}{n}\right)
\]

\begin{theorem}[{{\cite[Proposition 1]{AC}}}]
Suppose that $g \geq 3$. The rational Picard group $\operatorname{Pic}_\mathbb{Q}(\overline{\mathcal{M}}_{g,n})$ has a basis given by the classes $\lambda,\psi_1,...,\psi_n,\delta_{irr}$ and $\set{\delta_{i,S}}_{i,S}$.
\end{theorem}

The standard way to compute the canonical divisor on $\Mbar{g}{n}$ is to first compute it on the stack using Grothendieck-Riemann-Roch, and then study the ramification of the stack-to-coarse map to obtain the canonical class on the coarse moduli space.
For $g \geq 1$ and $g+n \geq 4$ the result is
\begin{equation}\label{eq:formula canonnical class}
    K_{\overline{\mathcal{M}}_{g,n}} = 13\lambda + \psi_1 + ... + \psi_n - 2 \delta - \delta_{1,\emptyset},
\end{equation}
as proven in \cite[Section 2]{HM} for $n=0$, and in \cite[Chapter 13, Theorem 7.16]{ACG} for the general case.\\

\subsection{Local description of $\Mbar{g}{n}$.}\label{subsec:local description Mgn}
Let $[C,p_1,...,p_n] \in \overline{\mathcal{M}}_{g,n}$ be a point in the coarse moduli space. The automorphism group $\operatorname{Aut}(C,p_1,...,p_n)$ acts on the space of first order infinitesimal deformations
\[
V=\coho{0}{C,\Omega_C \otimes \omega_C(p_1+...+p_n)}^*.
\]
There exists an analytic neighborhood of $[C,p_1,...,p_n]$ in the moduli space that is isomorphic to a neighborhood of the origin in the quotient
\[
\coho{0}{C,\Omega_C \otimes \omega_C(p_1+...+p_n)}^*/\operatorname{Aut}(C,p_1,...,p_n).
\]
We now recall some facts related to this quotient, that we will need later.
Let $\tilde{C}_1,...,\tilde{C}_\nu$ be the irreducible components of $C$ and let $\pi_\alpha : C_\alpha \lra \tilde{C}_\alpha$ be the normalization. Let $\Omega_\alpha$ be the cotangent sheaf on $C_\alpha$.
Pulling back differentials gives a morphism of sheaves on $C$,
\begin{equation}\label{eq:pull back of differentials}
    \Omega_C \otimes \omega_C\left(p_1+...+p_n\right) \lra \bigoplus_{\alpha=1}^\nu \pi_{\alpha*}\Omega_\alpha^{\otimes2}\left(\sum_Q Q\right),
\end{equation}
where $Q$ ranges over all special points of $C_\alpha$, i.e. points of $C_\alpha$ that are mapped to a node or marked point.
The morphism \eqref{eq:pull back of differentials} is an isomorphism away from the nodes of $C$.
If $P$ is a node of $C$ locally given by the equation $xy=0$, then $\Omega_C$ is generated around $P$ by the differentials $dx,dy$ subject to the relation $xdy+ydx = 0$, and $\omega_C\left(p_1+...+p_n\right)$ is locally generated by the section $\eta$ which is $\frac{dx}{x}$ on the branch $y=0$ and $-\frac{dy}{y}$ on the branch $x=0$.
From this we see that \eqref{eq:pull back of differentials} is surjective at $P$, with kernel generated by the torsion differential $\tau_P=ydx\otimes\eta = -xdy\otimes\eta$. Therefore we get a short exact sequence of sheaves
\begin{equation}\label{eq:SES of sheaves}
    0 \lra \bigoplus_{P \in C \text{ is a node}}\cc_P \lra \Omega_C \otimes \omega_C\left(p_1+...+p_n\right) \lra \bigoplus_{\alpha=1}^\nu \pi_{\alpha*}\Omega_\alpha^{\otimes 2}\left(\sum_Q Q\right) \lra 0,
\end{equation}
where the map $\cc_P \lra \Omega_C \otimes \omega_C\left(p_1+...+p_n\right)$ sends $a \in \cc$ to $a \tau_P$.
The element $\tau_P$ is called a smoothing parameter for $P$.
Viewing vectors in $\coho{0}{C,\Omega_C\otimes\omega_C(p_1,...,p_n)}$ as linear functionals on $V$, the subspace of $V$ where $\tau_P$ is zero parametrizes deformations of $(C,p_1,...,p_n)$ in which $P$ remains a node.

\subsection{The Reid-Shepherd-Barron-Tai criterion}
Let $G$ be a finite group that acts linearly on a complex vectorspace $V$ of dimension $d$. Let $g \in G$. Since $g$ has finite order, its action on $V$ can be diagonalized and the eigenvalues are roots of unity. Suppose $g$ has eigenvalues $e^{2\pi i r_1},...,e^{2\pi i r_d}$ where $0 \leq r_j < 1$.
The \textit{age} of $g$ on $V$ is defined as
\[
\operatorname{age}_V(g) = \operatorname{age}(g) = \sum_{j=1}^d r_j.
\]
An element $g \in G$ is called \textit{junior} if $0 < \operatorname{age}(g) < 1$ and it is called a \textit{quasi-reflection} if exactly one of the eigenvalues of $g$ is not $1$. The classical version of the Reid-Shepherd-Barron-Tai criterion states:

\begin{proposition}[{\cite[Theorem 3.3]{Ta}, \cite[Theorem 3.1]{Re} }]\label{Reid-Tai OG}
    Suppose that $G$ has no quasi-reflections.
    Then $V/G$ has canonical singularities if and only if $G$ has no junior elements.
\end{proposition}

A limitation of the Reid-Shepherd-Barron-Tai criterion is that the criterion gives no information about \textit{how bad} the non-canonical singularities are.
Even when $V/G$ has non-canonical singularities, some pluri-canonical forms will lift to a resolution of singularities.
In what follows, we describe a condition, due to \cite{CCM}, such that pluri-canonical forms on $V/G$ that satisfy the condition lift to a resolution of singularities.
The first step is to reduce to the case where $G$ is cyclic.

\begin{lemma}[{\cite[Proposition 3.1]{Ta}}]
    Let $\eta$ be a pluri-canonical form defined on a neighborhood of the origin in $V/G$. Then $\eta$ lifts to a resolution of singularities if and only if the pull-back of $\eta$ by $V/\langle g \rangle \lra V/G$ lifts to a resolution of singularities for all $g \in G$.
\end{lemma}
In particular, if $V/\langle g \rangle$ has canonical singularities for all $g \in G$, then $V/G$ has canonical singularities.

\begin{proposition}[{\cite[Corollary 5.5]{CCM}}]\label{Reid-Tai with vanishing}
    Let $g = \operatorname{diag}(e^{2\pi ir_1},...,e^{2\pi ir_d})$ be an element of finite order $k$ acting on $\cc^d$ and suppose that $\langle g \rangle$ has no quasi-reflections.
    Let $D_j \subseteq \cc^d$ be the locus where the $j$-th coordinate is zero.
    Let $\eta$ be an $m$-canonical form defined in a neighborhood of the origin of $\cc^d/\langle g \rangle$ and for $j=1,...,d$ let $b_j$ be the order of vanishing of $\eta$ along $D_j$.
    Then $\eta$ extends to a resolution of singularities if
    \begin{equation}\label{eq:inequality for lifting forms}
        \sum_{j = 1}^d (b_j+m)\set{lr_j} \geq m \text{ for } l = 1,...,k-1,
    \end{equation}
    where $\{x\}$ denotes the fractional part of a real number $x$.
\end{proposition}

Proposition \ref{Reid-Tai OG} and Proposition \ref{Reid-Tai with vanishing} assume that there are no quasi-reflections in $G$ and $\langle g \rangle$ respectively.
We now describe how to deal with quasi-reflection in the case where $G=\langle g \rangle$ is cyclic. Suppose $g = \operatorname{diag}(e^{2\pi ir_1},...,e^{2\pi ir_d})$ acts on $\cc^d$ with order $k$. Let $x_1,...,x_d$ be the standard coordinates on $\cc^d$. For $j \in \set{1,...,d}$ fixed, let $m_j$ be the minimal positive integer such that $g^{m_j}$ fixes $x_i$ for all $i \neq j$.
Then $H = \langle g^{m_1},...,g^{m_d} \rangle$ is the subgroup of $\langle g \rangle$ generated by all quasi-reflections. The quotient $\cc^d/H$ is isomorphic to affine space with coordinates $y_j = x_j^{\lambda_j}$ for $j=1,...,d$, where $\lambda_j = \frac{k}{m_j}$,
and $G/H$ acts linearly on $\cc^d/H$ without quasi-reflections.
More precisely, the generator $gH$ of $G/H$ acts on $\cc^d/H$ with eigenvalues $e^{2\pi i\lambda_1r_1},...,e^{2\pi i\lambda_dr_d}$.
We can then apply Proposition \ref{Reid-Tai OG} and Proposition \ref{Reid-Tai with vanishing} to study the lifting of pluri-canonical forms on
\[
\cc^d/G \simeq \left(\cc^d/H\right)/\left(G/H\right).
\]
The following examples illustrate the theory above. These examples will appear in the analysis of the singularities of $\Mbar{3}{n}$, when considering automorphism of curves with an elliptic tails with $j$-invariant $1728$ and $0$ respectively.
\begin{example}\label{ex:j=1728}
    Let $G$ be cyclic of order $4$, generated by $g = \operatorname{diag}(\pm i, - 1, \pm 1,...,\pm1)$. We will show that $\cc^d/G$ has canonical singularities. Note that $g^2 = \operatorname{diag}(-1,1,...,1)$ is a quasi-reflection. The element $\bar{g}$ acts on $\cc^d/\langle g^2 \rangle$ with eigenvalues $-1,-1,\pm1,...,\pm1$. Therefore the action of $G/\langle g^2 \rangle = \set{\bar{1},\bar{g}}$ on $\cc^d/\langle g^2 \rangle$ has no junior elements.
    By Proposition \ref{Reid-Tai OG}, $\cc^d/G$ has canonical singularities.
\end{example}

\begin{example}\label{ex:j=0}
    Let $G$ be cyclic of order $6$, generated by $g = \operatorname{diag}(\pm \zeta_6, \zeta_6^2, \pm 1,...,\pm1)$, where $\zeta_6$ is a primitive sixth root of unity.
    We will show that any $m$-canonical form $\eta$ on $\cc^d/G$ with order of vanishing at least $m$ along the divisor $D_1 = \set{x_1 = 0}$ extends to a resolution of singularities.
    We consider three cases.\\
    \textbf{Case 1:} $g = \operatorname{diag}(\zeta_6, \zeta_6^2, 1,...,1)$. Then $g^3 = \operatorname{diag}(-1,1,...,1)$ is a quasi-reflection. The element $\bar{g}$ acts on $\cc^d/\langle g^3 \rangle$ with eigenvalues $\zeta_6^2,\zeta_6^2,1,...,1$. Then $G/\langle g^3 \rangle$ has junior elements. However, condition \eqref{eq:inequality for lifting forms} is satisfied with $b_1 \geq m$. Hence $\eta$ lifts to a resolution of singularities.\\
    \textbf{Case 2:} $g = \operatorname{diag}(-\zeta_6, \zeta_6^2, -1,...,1)$. Then $g^3 = \operatorname{diag}(1,1,-1,...,1)$ is a quasi-reflection. The element $\bar{g}$ acts on $\cc^d/\langle g^3 \rangle$ with eigenvalues $-\zeta_6,\zeta_6^2,1,...,1$. Then $G/\langle g^3\rangle$ has no junior elements, and hence $\cc^d/\langle g^3 \rangle$ has canonical singularities.\\
    \textbf{Case 3:} In the remaining cases, one checks that $G$ has no junior elements, and therefore $\cc^d/G$ has canonical singularities.
\end{example}

The variants of the Reid-Shepherd-Barron-Tai criterion above, will allow us to show that certain pluri-canonical forms on $\Mbar{3}{n}$ lift to a resolution of singularities on the complement of a high codimension locus (for details, see the proof of Proposition \ref{prop:lifting forms local version}). 
The following result will allow us to deal with the remaining points.
\begin{proposition}[{{\cite[Appendix 1 to Section 1]{HM}}}]\label{prop:lifting forms HM version}
    Let $\eta$ be a pluri-canonical form defined in a neighborhood of the origin in $V/G$.
    Suppose that for all junior elements $g \in G$ there exists a vector $v \in V$ fixed by $g$ such that $\eta$ extends to a resolution of singularities of a neighborhood of $[v] \in V/G$. Then $\eta$ lifts to a resolution of singularities of a neighborhood of the origin of $V/G$.
\end{proposition}

\section{Singularities on $\Mbar{3}{n}$}\label{sec:lifting forms}

The goal of this section is to prove Proposition \ref{prop:condition to lift forms}.
The first step in the singularity analysis is to study the age of automorphisms of pointed curves acting on the space of first order infinitesimal deformations, in order to apply the Reid-Shepherd-Barron-Tai criterion. 
Following \cite[Section 1]{HM} we first study the age of automorphisms of smooth curves.
In the next lemma, we allow the automorphism to swap points in pairs. These pairs will correspond to pairs of inverse images of a node of a stable curve.

\begin{lemma}\label{lemma:age on smooth pointed curve}
    Let $[C,p_1,...,p_r,q_1,...,q_{2s}] \in \mathcal{M}_{g,r+2s}$ be a smooth $r+2s$-pointed curve of genus $g$ and suppose that $r+2s > 0$.
    Let $\varphi$ be a non-trivial automorphism of $C$ that fixes the points $p_i$ for $i=1,...,r$ and swaps the points $q_{2j-1},q_{2j}$ for $j=1,...,s$. Consider the age on the vectorspace $V = \coho{0}{C,\omega_C^{\otimes 2}\left(\sum_{i=1}^r p_i + \sum_{j=1}^{2s} q_j \right)}$. Then either $\operatorname{age}(\varphi) \geq 1$, or $s \geq 1$ and $\varphi$ has order $2$, or $r=0$ and $g+s \leq 2$, or $(C,p_1,...,p_r,q_1,...,q_{2s},\varphi)$ is as described in one of the rows of Table \ref{table: junior autos of smooth curves}.
\end{lemma}

\begin{table}[h]
    \centering
    \begin{tabular}{|c|c|c|c|c|c|c|}
    \hline
         Case number & $g$ & $r$ & $s$ & Description of $C$ & Description of $\varphi$ & Eigenvalues\\ \hline
         (1) & $1$ & $1$ & $0$ & / & $\operatorname{ord}(\varphi)=2$ & $1$ \\ \hline
         (2) & $1$ & $2$ & $0$ & / & $\operatorname{ord}(\varphi)=2$ & $1,-1$ \\ \hline
         (3) & $1$ & $1$ & $0$ & $j(C)=0$ & $\operatorname{ord}(\varphi)=3$ & $\zeta_3^2$ \\ \hline
         (4) & $1$ & $1$ & $0$ & $j(C)=0$ & $\operatorname{ord}(\varphi)=6$ & $\zeta_6^2$ \\ \hline
         (5) & $1$ & $1$ & $0$ & $j(C)=1728$ & $\operatorname{ord}(\varphi)=4$ & $-1$. \\ \hline
         (6) & $1$ & $2$ & $0$ & $j(C)=1728$ & $\operatorname{ord}(\varphi)=4$ & $-1,\zeta_4$. \\ \hline
         (7) & $2$ & $1$ & $0$ & / & hyperelliptic involution & $1,1,1,-1$. \\ \hline
         (8) & $3$ & $0$ & $0$ & $C$ is hyperelliptic & hyperelliptic involution & $1,1,1,1,1,-1$. \\ \hline
    \end{tabular}
    \caption{
    List of some non-trivial junior automorphisms of smooth pointed curves. The automorphism fixes $r$ marked points and swaps $s$ pairs of marked points.}
    \label{table: junior autos of smooth curves}
\end{table}

\begin{proof}
Consider the subspace  $W = \coho{0}{C,\omega_C^{\otimes 2}}$ of $V$. The action of $\varphi$ on $V$ restricts to an action on $W$ and hence $\operatorname{age}_W(\varphi) \leq \operatorname{age}_V(\varphi)$.
By the first proposition in \cite[Section 1]{HM} we have that $\operatorname{age}_W(\varphi) \geq 1$ or $(C,\varphi)$ is as described in one of the following cases:
\begin{itemize}
    \item[(A)] $g = 0$.
    \item[(B)] $g = 1$.
    \item[(C)] $g = 2$ and $\varphi$ is the hyperelliptic involution.
    \item[(D)] $g = 2$, $C$ is a double cover of an elliptic curve and $\varphi$ is the associated involution.
    \item[(E)] $g = 3$, $C$ is hyperelliptic and $\varphi$ is the hyperelliptic involution.
\end{itemize}
It is enough to prove the lemma for these cases only.\\

\textbf{Case (A):} $C$ is rational.\\
The pointed curve in question is stable, thus $r+2s \geq 3$. Since $\varphi^2$ fixes $p_1,...,p_r,q_1,...,q_{2s}$ we conclude that $\varphi$ has order $2$. This implies that $r \leq 2$ and therefore $s \geq1$. We are done in this case.\\

\textbf{Case (B):} $C$ has genus $1$.\\
We will view $C$ as a quotient of the complex plane $\cc$ by a lattice $\Lambda$.
We split in several sub-cases.\\

\textbf{Case (B1):} $\varphi$ is a translation.\\
In this case $\varphi$ has no fixed points. It follows that $r=0$. Hence $s \geq 1$ and $\varphi$ has order $2$.\\

\textbf{Case (B2):} $\varphi$ has a fixed point on $C$ and has order $2$.\\
If $s \geq 1$ then we are done. So assume $s = 0$. Then $r\geq1$. If we choose $p_1$ as the origin then $\varphi$ is given by $z \longmapsto -z$. Suppose $R \in C$ is a fixed point of $\varphi$, different from $p_1$. There exists an eigenbasis for $\varphi$ of the form
\[
1,f_R \in \coho{0}{C,\mathcal{O}_{C}(p_1+R)}.
\]
Then $f_R$ must have a simple pole at $p_1$. We can look at a Laurent expansion of $f_R$ around the origin to conclude that $\varphi^*f_R = -f_R$.
We obtain an eigenbasis
\[
dz^{\otimes2},f_{p_2}dz^{\otimes2},...,f_{p_r}dz^{\otimes2} \in V,
\]
with eigenvalues $1,-1,...,-1$. Hence $\operatorname{age}(\varphi) = \frac{r-1}{2}$.
The only cases for which $\operatorname{age}\varphi < 1$ are the cases $(1)$ and $(2)$ of Table \ref{table: junior autos of smooth curves}. The eigenvalues of $\varphi$ are $1$ and $1,-1$ respectively.\\

\textbf{Case (B3):} $\varphi$ has a fixed point on $C$ and has order $3$.\\
In this case $\varphi$ can not swap a pair of points. Thus $s=0$ and $r \geq 1$. If we choose $p_1$ as the origin, then $\varphi$ is given by $z \longmapsto \zeta_3z$, where $\zeta_3$ is a primitive third root of unity. Suppose $r \geq 2$.
Let $1,f \in \coho{0}{C,\mathcal{O}_C(p_1+p_2)}$ be an eigenbasis of $\varphi$. Then $f$ has a simple pole at $p_1$ and looking at the Laurent expansion of $f$ around $p_1$, we see that $\varphi^*f = \zeta_3^{-1}f$. Then 
\[
dz^{\otimes2},fdz^{\otimes2} \in \coho{0}{C,\omega_C\left(p_1+p_2\right)}
\]
is an eigenbasis for $\varphi$ with eigenvalues $\zeta_3^2,\zeta_3$. This gives a contribution of $\frac{1}{3} + \frac{2}{3} = 1$. Thus $\operatorname{age}(\varphi)
\geq 1$ if $r \geq 2$. The only case left is case $(3)$ of Table \ref{table: junior autos of smooth curves}, with eigenvalues $\zeta_3^2$.\\

\textbf{Case (B4):} $\varphi$ has a fixed point $R_1 \in C$ and has order $6$.\\
Then $\varphi$ has only $R_1$ as fixed point, while $\varphi^2$ has two more fixed points $R_2$ and $R_3 = \varphi R_2$.
The set $\set{p_1,...,p_r,q_1,...,q_{2s}}$ equals $\set{R_1}$, $\set{R_2,R_3}$ or $\set{R_1,R_2,R_3}$. If we choose $R_1$ as the origin, then $\varphi$ is given by $z \longmapsto \zeta_6z$ where $\zeta_6$ is a primitive $6^{th}$ root of unity.\\ Just as in the case (B3) we can find a nonzero function
\[
f \in \coho{0}{C,\mathcal{O}_{C}(R_1+R_2)}
\]
such that $(\varphi^2)^*f = \zeta_6^{-2}f$. Set
\[
h_1 = f + \zeta_6^{-2}\varphi^*f, h_2 = f + \zeta_6^{-5}\varphi^*f.
\]
Then we obtain an eigenbasis
\[
dz^{\otimes2},h_1dz^{\otimes2},h_2dz^{\otimes2} \in \coho{0}{C,\omega_C^{\otimes2}(R_1+R_2+R_3)}
\]
of $\varphi$ with eigenvalues $\zeta_6^2,\zeta_6^4,\zeta_6$. Therefore if $\set{p_1,...,p_r,q_1,...,q_{2s}} = \set{R_1,R_2,R_3}$ then $\operatorname{age}(\varphi) \geq 1$. If $\set{p_1,...,p_r,q_1,...,q_{2s}} = \set{R_2,R_3}$ then $r=0$ and $g+s \leq 2$.
The only the case left is case $(4)$ of Table \ref{table: junior autos of smooth curves}, with eigenvalues $\zeta_6^2$.\\

\textbf{Case (B5):} $\varphi$ has a fixed point on $C$ and has order $4$.\\
Then $\varphi$ has exactly two fixed points $R_1,R_2 \in C$ and $\varphi^2$ has two more fixed points $R_3$ and $R_4=\varphi R_3$. The set $\set{p_1,...,p_r,q_1,...,q_{2s}}$ is one of the following: $\set{R_1},\set{R_1,R_2},\set{R_1,R_3,R_4}$ or $\set{R_1,R_2,R_3,R_4}$ (this is true up to swapping the roles of $R_1$ and $R_2$). If we choose $R_1$ as the origin, then $\varphi$ is given by $z \longmapsto \zeta_4z$ where $\zeta_4=\pm i$ is a primitive fourth root of unity.\\
Just as in case (B2), we can find a basis $1,f \in \coho{0}{C,\mathcal{O}_C(R_1+R_3)}$ such that $(\varphi^2)^*f = -f$. Set
\[
h_1 = f + \zeta_4 \varphi^*f, h_2 = f - \zeta_4 \varphi^*f.
\]
Then we obtain an eigenbasis
\[
dz^{\otimes2},h_1dz^{\otimes2},h_2dz^{\otimes2} \in \coho{0}{C,\omega_C^{\otimes2}\left(R_1 + R_3 + R_4\right)}
\]
for $\varphi$ with eigenvalues $-1,\zeta_4$ and $-\zeta_4$. Therefore if $\set{p_1,...,p_r,q_1,...,q_{2s}}$ is $\set{R_1,R_3,R_4}$ or $\set{R_1,R_2,R_3,R_4}$, then $\operatorname{age}(\varphi) \geq \frac{1}{2} + \frac{1}{4} + \frac{3}{4} = \frac{3}{2} \geq 1$.
If $\set{p_1,...,p_r,q_1,...,q_{2s}} = \set{R_3,R_4}$ then $r=0$ and $g+s \leq 2$.
We are left with only the cases $\set{p_1,...,p_r,q_1,...,q_{2s}} = \set{R_1}$ and $\set{p_1,...,p_r,q_1,...,q_{2s}} = \set{R_1,R_2}$, which correspond to $(5)$ and $(6)$ of Table \ref{table: junior autos of smooth curves}. We can argue as in case (B2) or (B3) to compute the eigenvalues in these cases.\\

\textbf{Case (C):} $C$ has genus $2$ and $\varphi$ is the hyperelliptic involution.\\
If $s \geq 1$ or $s=0$ and $r=0$ then we are done, so suppose $s=0$ and $r \geq 1$.
The hyperelliptic involution $\varphi$ acts trivially on $\coho{0}{C,\omega_C^{\otimes 2}}$, essentially because the line bundle $\omega_C^{\otimes2}$ induces the hyperelliptic involution $C \lra \pp^1$. 
For each $j \in \{1,...,r\}$, we can find a differential $\eta_j \in \coho{0}{C,\omega_C^{\otimes2}(p_j)}\backslash \coho{0}{C,\omega_C^{\otimes2}}$ that is also an eigenvector for $\varphi$.
By looking at a Laurent-expansion of $\eta_j$ around $p_j$ we see that the corresponding eigenvalue is $-1$.
This shows that the eigenvalues of $\varphi$ on $V$ are $1,1,1,-1,...,-1$ where $-1$ is repeated $r$ times. The only case left is $(7)$ from Table \ref{table: junior autos of smooth curves}.\\

\textbf{Case (D):} $C$ has genus $2$, admits a double cover $\pi:C \lra E$ of an elliptic curve $E$ and $\varphi$ is the associated involution.\\
If $s \geq 1$ or $s=0$ and $r=0$ then we are done, so suppose $s=0$ and $r \geq 1$. Then we can use the same argument as in case (C), together with the fact that $\varphi$ acts non-trivially on $\coho{0}{C,\omega_C^{\otimes2}}$ (see \cite[Chapter 11, Proposition 4.11]{ACG}) to conclude that $\operatorname{age}_V(\varphi) \geq 1$.

\textbf{Case (E):} $C$ has genus $3$, is hyperelliptic and $\varphi$ is the hyperelliptic involution.\\
If $s \geq 1$ then we are done. So suppose $s=0$.
We can find explicit equations for $C$: the curve is given by adding two points $\infty_1,\infty_2$ to the affine curve
\[
V(y^2-F(x)) \subset \mathbb{A}^2
\]
where $F \in \cc[x]$ has degree $8$ without double roots and $F(0) \neq 0$. In this case we have the explicit eigenbasis
\[
\frac{dx^{\otimes2}}{y},\frac{dx^{\otimes2}}{y^2},\frac{xdx^{\otimes2}}{y^2},\frac{x^2dx^{\otimes2}}{y^2},\frac{x^3dx^{\otimes2}}{y^2},\frac{x^4dx^{\otimes2}}{y^2} \in \coho{0}{C,\omega_C^{\otimes2}}
\]
for $\varphi$ with eigenvalues $-1,1,1,1,1,1$. If $r \geq 1$ then we conclude that $\operatorname{age}(\varphi) \geq 1$, similarly to case $(C)$ and $(D)$. The only case left is case $(8)$ of Table \ref{table: junior autos of smooth curves}.
\end{proof}

\begin{proposition}\label{prop:analysis of eigenvalues}
    Let $[C,p_1,...,p_n] \in \overline{\mathcal{M}}_{3,n}$ and let $\varphi$ be a non-trivial automorphism of the pointed curve $(C,p_1,...,p_n)$. Consider the age of $\varphi$ on the vectorspace 
    \[
    V=\coho{0}{C,\Omega_C \otimes \omega_C(p_1+...+p_n)}.
    \]
    Then either $\operatorname{age}(\varphi) \geq 1$, or $\varphi$ lifts to a deformation of $(C,p_1,...,p_n)$ with fewer nodes, or $n=0$ and $(C,\varphi)$ is hyperelliptic, or $C = C' \cup E$ has an elliptic tail $E$ without marked points, $\restr{\varphi}{C'} = \operatorname{id}$, and $\restr{\varphi}{E}$ has order $2,4$ or $6$. Furthermore, in the elliptic tail case, $\varphi$ has at most $2$ non-trivial eigenvalues on $V$.
\end{proposition}
\begin{proof}
Suppose $C$ is smooth. By Proposition \ref{lemma:age on smooth pointed curve}, we have that either $\operatorname{age}(\varphi) \geq 1$, or $n=0$ and $(C,\varphi)$ is hyperelliptic.
Thus we may assume that $C$ is not smooth.
We will prove the proposition by splitting in many cases.
Let $\tilde{C}_1,...,\tilde{C}_\nu$ be the irreducible components of $C$ and let $C_\alpha \lra \tilde{C}_\alpha$ be the normalization. Let $g_\alpha$ be the genus of $C_\alpha$, let $\delta_\alpha$ the number of special points of $C_\alpha$, and let $\delta$ be the number of nodes of $C$. Note that by stability of $C$ we have $2g_\alpha - 3 + \delta_\alpha \geq 0$ for all $\alpha =1,...,\nu$. From the equality
\[
3 = g = \sum_\alpha g_\alpha + \delta + 1 - \nu
\]
we deduce that
\begin{equation}\label{eq:combinatorial bound on number of components}
    \nu  \leq \nu + \sum_\alpha \Delta_\alpha = 2g-2 + n = 4 + n,
\end{equation}
where $\Delta_\alpha = 2g_\alpha - 3 + \delta_\alpha$.
Note also that each component of $C$ has a node, because $C$ is not smooth. In particular $\delta_\alpha \geq 1$ for each $\alpha$.
A fundamental tool in the analysis of the eigenvalues of $\varphi$ on $V=\coho{0}{\Omega_C \otimes \omega_C\left(p_1+...+p_n\right)}$ is the short exact sequence
\begin{equation}\label{eq:fundamental short exact sequence}
    0 \lra \bigoplus_P \coho{0}{C,\cc_P} \lra \coho{0}{C,\Omega_C \otimes\omega_C\left(p_1+...+p_n\right)} \lra \bigoplus_\alpha \coho{0}{C,\Omega_\alpha^{\otimes 2}\left(\sum_Q Q\right)} \lra 0,
\end{equation}
obtained from \eqref{eq:SES of sheaves} by taking global sections.
The eigenvalues of $\varphi$ on $V$ are the union of the eigenvalues of $\varphi$ on the left and right terms of \eqref{eq:fundamental short exact sequence}.\\
Following \cite{HM} we start with a study of the eigenvalues of $\varphi$ on the vectorspace $\bigoplus_P\coho{0}{C,\cc_P}$.
Let $P \in C$ be a node and let $m = m_P$ be the length of the orbit of $P$ under $\varphi$.
This means that $P,\varphi P,...,\varphi^{m-1} P$ are distinct, but $\varphi^m P = P$. Then $\varphi^m$ acts on $\coho{0}{C,\cc_P} \simeq \cc$. Take a nonzero element $e \in \coho{0}{C,\cc_P}$, corresponding to a smoothing parameter for $P$.\\
Suppose that $\varphi^m$ acts as the identity on $\coho{0}{C,\cc_P}$. Consider the nonzero vector
\[
\bar{e} = e + \varphi e + \cdots + \varphi^{m-1}e \in \bigoplus_{P'} \coho{0}{C,\cc_{P'}} \subseteq \coho{0}{C,\Omega_C \otimes \omega_C\left(p_1+...+p_n\right)}.
\]
Then $\varphi \bar{e} = \bar{e}$ and therefore $\bar{e}$ induces a deformation with fewer nodes to which $\varphi$ lifts. So in this case we are done.\\
From now on we may therefore assume that $\varphi^{m_P}$ acts non-trivially on $\coho{0}{C,\cc_P}$ for any node $P$. In this case we have
\[
\varphi^{m}e = \zeta^{ml}e
\]
for some $l \in \set{1,...,\frac{N}{m}-1}$, where $\zeta = e^{2\pi i/N}$ and $N = \operatorname{ord}(\varphi)$.
For each $a \in \{0,...,N-1\}$, consider the nonzero vector
\[
\bar{e}_a = \sum_{j=0}^{m-1} \zeta^{-aj}\varphi^je \in \bigoplus_{j=0}^{m-1} \coho{0}{C,\cc_{\varphi^jP}}
\]
We compute
\[
\begin{array}{rl}
    \varphi \bar{e}_a = & \sum_{j=0}^{m-1} \zeta^{-aj}\varphi^{j+1}e \\
    = & \sum_{j=1}^{m-1} \zeta^{-a(j-1)}\varphi^{j}e  + \zeta^{-a(m-1)} \varphi^{m}e\\
    = & \zeta^{a}\left( \sum_{j=1}^{m-1} \zeta^{-aj}\varphi^{j}e  + \zeta^{-am} \varphi^{m}e \right)\\
    = & \zeta^a \left( \zeta^{(-a+l)m}e + \sum_{j=1}^{m-1} \zeta^{-aj}\varphi^{j}e \right).
\end{array}
\]
Therefore $\bar{e}_a$ is an eigenvector of $\varphi$ with eigenvalue $\zeta^a$ whenever $a \equiv l \pmod{\frac{N}{m}}$, i.e. whenever $a \in \set{l,l+\frac{N}{m},l+2\frac{N}{m},...,l+(m-1)\frac{N}{m}}$. Therefore the eigenvalues of $\varphi$ on $\bigoplus_{j=0}^{m-1}\coho{0}{C,\cc_{\varphi^jP}}$ contribute
\[
\frac{ml}{N} + \frac{m-1}{2} \geq \frac{m}{N} + \frac{m-1}{2}
\]
to $\operatorname{age}(\varphi)$, where $\tilde{l}$ is the residue of $l$ modulo $\frac{m}{N}$.
This implies that $\operatorname{age}(\varphi) \geq 1$ when $m \geq 3$. If $m_P=m_{P'}=2$ for two nodes $P,P'$ that are on different orbits, then we also get $\operatorname{age}(\varphi) \geq 1$. The only cases left to consider are:
\begin{itemize}
    \item[(a)] $\varphi$ fixes all nodes of $C$;
    \item[(b)] $\varphi$ swaps two nodes of $C$, but fixes all other nodes of $C$ (in this case we have a contribution of at least $\frac{1}{2}$ from these two nodes).
\end{itemize}
Now we investigate the eigenvalues of $\varphi$ on $\bigoplus_\alpha \coho{0}{C_\alpha,\Omega_\alpha^{\otimes 2}\left(\sum_Q Q\right)}$.\\
Let $\alpha \in \set{1,...,\nu}$. Let $m = m_\alpha$ be the length of the orbit of $C_\alpha$ under $\varphi$.
Then $\varphi^m$ acts on $\coho{0}{C_\alpha,\Omega_\alpha^{\otimes 2}\left(\sum_Q Q\right)}$.
The space $\coho{0}{C_\alpha,\Omega_\alpha^{\otimes 2}\left(\sum_QQ\right)}$ has dimension $ 3g_\alpha - 3 + \delta_\alpha = g_\alpha + \Delta_\alpha$ by Riemann-Roch.
If $\varphi^m$ has eigenvalues $\zeta^{ml_i}$ ($i \in \set{1,...,g_\alpha+\Delta_\alpha}$ and $l_i \in \set{0,...,\frac{N}{m}-1}$) on $\coho{0}{C_\alpha,\Omega_\alpha^{\otimes 2}\left(\sum_Q Q\right)}$ then by essentially the same computations as before, we get that $\varphi$ has eigenvalues
\[
\zeta^{l_i + j\frac{m}{n}}, i \in \set{1,...,g_\alpha+\Delta_\alpha}, j \in \set{0,...,m-1}
\]
on $\bigoplus_{j=0}^{m-1}\coho{0}{\varphi^jC_\alpha,\Omega_{\varphi^jC_\alpha}^{\otimes 2}\left(\sum_Q Q\right)}$. This gives a contribution of at least
\[
(g_\alpha+\Delta_\alpha) \frac{m-1}{2}
\]
to $\operatorname{age}(\varphi)$.\\
Suppose $m\geq2$ and $g_\alpha + \Delta_\alpha \geq 2$. Then $\operatorname{age}(\varphi) \geq 1$.\\
Suppose $m \geq 2, g_\alpha + \Delta_\alpha=1$ and case $(b)$ holds. Then also $\operatorname{age}(\varphi) \geq 1$, because case $(b)$ gives a contribution $\frac{1}{2}$ and $(g_\alpha+\Delta_\alpha )\frac{m-1}{2} \geq \frac{1}{2}$.\\
Suppose $m \geq 2, g_\alpha + \Delta_\alpha =1$ and case $(a)$ holds. Then $\varphi \tilde{C_\alpha}$ passes through all the nodes on $\tilde{C}_\alpha$ and vice versa (because $\varphi$ fixes all nodes.) It follows that $C$ has only two irreducible components $\tilde{C}_\alpha$ and $\tilde{C}_\beta$. From \eqref{eq:combinatorial bound on number of components} we get
\[
4 + n = \nu + \Delta_\alpha + \Delta_\beta \leq \nu + g_\alpha + \Delta_\alpha + g_\beta + \Delta_\beta = 2 + 1 + 1 = 4.
\]
The only possibility is that $n = 0$, $g_\alpha = g_\beta =0$ and $\delta_\alpha = \delta_\beta = 4$. Since $\varphi$ fixes all nodes, $\varphi^2$ fixes all inverse images of nodes. Since there are $4$ such points on each component of $C$, and both components are rational, we see that $\varphi^2 = \operatorname{id}$.
Now take a node $P \in C$. Around $P$ we can find coordinates such that the curve $C$ is given by the equation $xy = 0$ and $\varphi$ is given by
\[
\varphi^*x=y, \varphi^*y=x.
\]
Then $\varphi$ acts trivially on a smoothing parameter for $P$.
Thus $\varphi$ extends to a deformation of $C$ with fewer nodes. This finishes the proof in the case that $m \geq 2$ and $g_\alpha+\Delta_\alpha = 1$.\\
Suppose that $m \geq 2$, $g_\alpha + \Delta_\alpha = 0$ and all nodes on $\tilde{C}_\alpha$ are fixed by $\varphi$. Then by the same argument, $C$ has only two components $\tilde{C}_\alpha$ and $\tilde{C}_\beta$ swapped by $\varphi$. But then
\[
4 + n = \nu + \Delta_\alpha + \Delta_\beta \leq \nu + g_\alpha + \Delta_\alpha + g_\beta + \Delta_\beta = 2 + 0 + 0 = 2.
\]
This is a contradiction, so this case doesn't occur (another way to get a contradiction: the curve described above has genus less than $3$).\\
Suppose that $m \geq 2$, $g_\alpha + \Delta_\alpha = 0$ and some node $P$ on $\tilde{C}_\alpha$ is not fixed by $\varphi$. We must have $g_\alpha =0, \delta_\alpha = 3$. Denote the image $\varphi \tilde{C}_\alpha$ by $\tilde{C}_\beta$. If $C_\alpha$ self-intersects at $P$, then $\varphi P$ is not on $\tilde{C}_\alpha$ (because $\delta_\alpha =3$). Then there is some node $P'$ on $\tilde{C}_\alpha$ that is fixed by $\varphi$. The only option is that $C_\alpha$ and $C_\beta$ intersect at this point. Since $C$ is connected, we must then have that $C$ has only two irreducible components. Then $C$ has genus $2$, a contradiction.\\
Thus $C_\alpha$ meets another component $C_\gamma$ at $P$ ($\beta$ and $\gamma$ may or may not be the same).
Since $\delta_\alpha = 3$, there is a node $P'$ on $\tilde{C}_\alpha$ different from $P$ and $\varphi P$. Hence $P'$ is fixed by $\varphi$. The only option is that $C_\alpha$ and $C_\beta$ intersect at $P'$ and $\varphi \tilde{C}_\beta = \tilde{C}_\alpha$.
Then $\varphi^2$ fixes all the inverse images of the nodes on $C_\alpha$ and $C_\beta$. It follows that $\varphi^2$ is the identity on $\tilde{C}_\alpha \cup \tilde{C}_\beta$. And now we can find a deformation of $C$ with fewer nodes to which $\varphi$ lifts, since $\varphi$ acts trivially on the smoothing parameter for $P$.\\

We have now proven the theorem for all cases where $m = m_\alpha \geq 2$, so from now on we may assume that $m_\alpha = 1$ for all $\alpha$. In other words, $\varphi$ sends all components of $C$ to themselves. Denote by $\varphi_\alpha$ the automorphism of $C_\alpha$ induced by $\varphi$ and set $N_\alpha = \operatorname{ord}\varphi_\alpha$.\\
By the first proposition of \cite[Section 1]{HM} the age of $\varphi$ on $W = \coho{0}{C_\alpha,\Omega_\alpha^{2}}$ is at least $1$, except in the following cases:
\begin{itemize}
    \item[(A)] $\varphi_\alpha = \operatorname{id}$.
    \item[(B)] $g_\alpha = 0$.
    \item[(C)] $g_\alpha = 1$.
    \item[(D)] $g_\alpha = 2$ and $\varphi_\alpha$ is the hyperelliptic involution.
    \item[(E)] $g_\alpha = 2$, $C_\alpha$ is a double cover of an elliptic curve and $\varphi_\alpha$ is the associated involution.
    \item[(F)] $g_\alpha = 3$, $C_\alpha$ is hyperelliptic and $\varphi$ is the hypperelliptic involution.
\end{itemize}
Since $W$ is a subspace of $V$, we have $\operatorname{age}_V(\varphi) \geq \operatorname{age}_W(\varphi)$. We may therefore assume that for each $\alpha$, one of the statements $(A)-(F)$ is satisfied.\\

Note that $\varphi^2$ fixes all nodes of $C$ and therefore $\varphi^4$ fixes all special points of $C_\alpha$.
In case $(B)$ this implies $\varphi_\alpha^4 = \operatorname{id}$ because $\delta_\alpha \geq 3$.
In case $(C)$ either $\varphi_\alpha$ fixes some point or $\varphi_\alpha$ is a translation with respect to the group structure on $C_\alpha$. In the first case $N_\alpha \in \set{1,2,3,4,6}$. In the second case $\varphi_\alpha^4 = \operatorname{id}$ because $\varphi_\alpha^4$ is again a translation and has a fixed point.
In cases $(D),(E)$ and $(F)$ we have $N_\alpha = 2$.\\
Thus in any case we have $N_\alpha \in \set{1,2,3,4,6}$.\\
We will now complete the proof in case $(b)$, i.e. when $\varphi$ moves some node $P \in \tilde{C}_\alpha$. We split this in two sub-cases:
\begin{itemize}
    \item[(b1)] The component $\tilde{C}_\alpha$ self-intersects in $P$.
    \item[(b2)] The component $\tilde{C}_\alpha$ intersects another component $\tilde{C}_\beta$ in $P$.
\end{itemize}

\textbf{Case (b1):} Recall that the eigenvalues of $\varphi$ on $\coho{0}{C,\cc_P} \oplus \coho{0}{C,\cc_{\varphi P}}$ contribute
\[
\frac{m_Pl}{N} + \frac{m_P-1}{2} = \frac{2l}{N} + \frac{1}{2}
\]
to $\operatorname{age}(\varphi)$, where $\varphi^2$ acts as multiplication by $\zeta^{2l}$ on $\coho{0}{C,\cc_P}$, and $N = \operatorname{ord}(\varphi)$.
Note that $\varphi^{N_\alpha}$ acts as the identity on a neighborhood of $P$ and therefore $\frac{N}{N_\alpha}$ divides $l$. It follows that the contribution is at least
\[
\frac{2}{N_\alpha} + \frac{1}{2}.
\]
If $N_\alpha \leq 4$, then we obtain $\operatorname{age}(\varphi) \geq 1$.
If $N_\alpha = 6$, then $C_\alpha$ has genus $1$, and $\varphi$ acts as a third root of unity on $dz^{\otimes 2} \in \coho{0}{C_\alpha,\Omega_\alpha^{\otimes2}}$. This adds a contribution of at least $\frac{1}{3}$ to $\operatorname{age}(\varphi)$. Hence in this case
\[
\operatorname{age}(\varphi) \geq \frac{2}{N_\alpha} + \frac{1}{2} + \frac{1}{3} = \frac{7}{6} \geq 1.
\]

\textbf{Case (b2):} In this case we similarly get a contribution of at least
\[
\frac{2}{\operatorname{lcm}(N_\alpha,N_\beta)} + \frac{1}{2}
\]
from $\coho{0}{C,\cc_P} \oplus \coho{0}{C,\cc_{\varphi P}}$ to $\operatorname{age}(\varphi)$. If $\operatorname{lcm}(N_\alpha,N_\beta) \leq 4$, then $\operatorname{age}(\varphi) \geq 1$.
Suppose that $\operatorname{lcm}(N_\alpha,N_\beta) > 4$. Then (possibly after swapping $\alpha$ and $\beta$) we may assume that $N_\alpha \in \set{3,6}$.
Then $C_\alpha$ has genus $1$ and $\varphi$ acts as a third root of unity on $dz^{\otimes 2} \in \coho{0}{C_\alpha,\Omega_\alpha}$. Since $\operatorname{lcm}(N_\alpha,N_\beta) \leq 12$ we get
\[
\operatorname{age}(\varphi) \geq \frac{2}{\operatorname{lcm}(N_\alpha,N_\beta)} + \frac{1}{2} + \frac{1}{3} \geq \frac{2}{12} + \frac{1}{2} + \frac{1}{3} = 1.
\]
This finishes the proof in case $(b)$.\\

We are left to prove the proposition in the case that $\varphi$ fixes all components of $C$ and all nodes of $C$.
Fix a component $C_\alpha$ of $C$.
Then $\varphi_\alpha^2$ fixes all special points of $C_\alpha$.
Let $x_1,...,x_r,y_1,...,y_{2s}$ be the special points of $C_\alpha$, ordered so that  $x_1,...,x_r$ are fixed by $\varphi_\alpha$ and $y_1,...,y_{2s}$ are swapped in pairs.
If the age of $\varphi_\alpha$ on $\coho{0}{C,\omega_C^{\otimes 2}\left(\sum_{i=1}^r p_i + \sum_{j=1}^{2s} q_j \right)}$ is at least $1$, then we are done.
If $s \geq 1$ and $\varphi_\alpha$ has order $2$, then we can find a deformation of $(C,p_1,...,p_n)$ to which $\varphi$ lifts, and in which the node corresponding to $y_1,y_2$ disappears.
If $r=0$ then $C$ has only one components of geometric genus $g_\alpha$ with $s$ nodes. Then $C$ has genus $g_\alpha + s$, so $g_\alpha + s = 3$. 
By lemma \ref{lemma:age on smooth pointed curve} applied to $(C_\alpha,x_1,...,x_r,y_1,...,y_{2s})$ and $\varphi_\alpha$, we only have to consider the case that for each $\alpha$ either $\varphi_\alpha = \operatorname{id}$ (we call this case $(0)$), or the data $(C_\alpha,x_1,...,x_r,y_1,...,y_{2s},\varphi_\alpha)$ is as described in one of the cases $(1)-(8)$ in Table \ref{table: junior autos of smooth curves}. If case $(8)$ occurs, then $\delta_\alpha = r + 2s = 0$. Then $C=C_\alpha$ is smooth, but we already excluded this case.\\

Before continuing, we note that each node $P$ contributes at least $\frac{1}{n_P}$ to $\operatorname{age}(\varphi)$, where $n_P$ is the order of $\varphi$ when restricted to the components of $C$ that intersect at $P$. In all the cases $(1)-(7)$ we have $s=0$ and hence $r = \delta_\alpha$. Each component $C_\alpha$ also has an inverse image $x_i$ of a node $P \in C$, and we reorder the points $x_1,...,x_r$ such that $x_1$ is mapped to the node $P$.\\

Suppose that case $(7)$ occurs for $C_\alpha$. Let $C_\beta$ be the other component of $C$ that intersects $C_\alpha$ at $P$. Case $(7)$ contributes $\frac{1}{2}$ and $P$ contributes $\frac{1}{n_P} = \frac{1}{\operatorname{lcm}(2,N_\beta)}$, so if $N_\beta \leq 2$ then $\operatorname{age}(\varphi) \geq 1$. So suppose $N_\beta > 2$. Then $(C_\beta,\varphi_\beta)$ is as described in one of the cases $(3)-(6)$. In all these cases $C_\beta$ contributes at least $\frac{1}{3}$ and $n_P \leq 6$. Thus
\[
\operatorname{age}(\varphi) \geq \frac{1}{2} + \frac{1}{3} + \frac{1}{6} = 1.
\]
Thus we may exclude case $(7)$.

Suppose that case $(6)$ occurs. Then $C_\alpha$ contributes $\frac{3}{4}$ to $\operatorname{age}(\varphi)$. Let $C_\beta$ be the component that intersects $C_\alpha$ at $P$ (we don't exclude the case $\alpha =\beta$).
If $n_{P} \leq 4$ then $\operatorname{age}(\varphi) \geq 1$. Suppose $n_P > 4$. The only option is that $(C_\beta,\varphi_\beta)$ is as described in case $(3)$ or $(4)$. But then $C_\beta$ contributes at least $\frac{1}{3}$, thus $\operatorname{age}(\varphi) \geq \frac{3}{4} + \frac{1}{3} \geq 1$.
Thus we may exclude case $(6)$.

Suppose that case $(2)$ occurs. Then $C_\alpha$ contributes $\frac{1}{2}$. Let $C_\beta$ be the other component passing through $P$. Note that $\beta \neq \alpha$ because otherwise $C$ would have genus $2$. In all the cases $(0),(1),(2),(3),(4),(5)$ for $C_\beta$ we get that
\[
\operatorname{age}(\varphi) \geq \frac{1}{2} + \frac{1}{\operatorname{lcm}(2,N_\beta)} + \text{ contribution of } C_\beta \geq 1.
\]
Thus we may exclude case $(2)$.\\
Suppose that one of the cases $(1),(3),(4)$ or $(5)$ occurs. Let $C_\beta$ be the other component passing through $P$. If $(C_\beta,\varphi_\beta)$ is as described in case $(1),(3),(4)$ or $(5)$, then we get that $C = \tilde{C}_\alpha \cup \tilde{C}_\beta$ has genus $2$, a contradiction. The only case left for $\beta$ is case $(0)$, i.e. $\varphi_\beta = \operatorname{id}$. Now one checks that the contribution of $C_\alpha$ and the contribution $\frac{1}{n_P}$ add up to at least $\frac{1}{2}$.
Either $\varphi = \operatorname{id}$ on $C \backslash \tilde{C}_\alpha$ or one of the cases $(1),(3),(4)$ or $(5)$ occurs again. In the latter case we get another contribution of at least $\frac{1}{2}$ and then $\operatorname{age}(\varphi) \geq 1$. In the former cases $P$ is the only node of $C$ (if there is another node $P'$ then $\varphi$ acts trivially on $\coho{0}{C,\cc_{P'}}$ which we assumed not to be the case).
Thus the only case left is that $C$ has two irreducible components $C_\alpha,C_\beta$ that are smooth, intersect only at $P$, and have genus $1$ and $2$ respectively, $\varphi_\beta = \operatorname{id}$, and $C_\alpha$ has no marked points. Then $E = C_\alpha$ is an elliptic tail.
If $\operatorname{ord}(\varphi_\alpha) = 3$, then $\varphi$ has eigenvalues $\zeta_3$ and $\zeta_3^2$ on $\coho{0}{C,\cc_P}$ and $\coho{0}{C_\alpha,\Omega_\alpha^2}$ respectively, which implies that $\operatorname{age}(\varphi) \geq 1$.
To finish the proof, we observe that if $\operatorname{ord}\varphi_\alpha \in \set{2,4,6}$, then by the explicit list of eigenvalues described in Table \ref{table: junior autos of smooth curves} we see that $\varphi$ has at most $2$ non-trivial eigenvalues on $V$.
\end{proof}

\begin{proposition}\label{curves admitting a junior automorphism}
    Suppose that $n \geq 1$. Let $[C,p_1,...,p_n] \in \overline{\mathcal{M}}_{3,n}$ and let $\varphi$ be a non-trivial automorphism of the $n$-pointed curve $(C,p_1,...,p_n)$. Consider the age of $\varphi$ on the vectorspace $V=\coho{0}{C,\Omega_C \otimes \omega_C(p_1+...+p_n)}$. If $\varphi$ is junior, then $[C,p_1,...,p_n] \in \Delta_{1,\emptyset}$.
\end{proposition}
\begin{proof}
    Suppose that $[C,p_1,...,p_n] \not\in \Delta_{1,\emptyset}$. We show that $\operatorname{age}(\varphi) \geq 1$ by induction on the number of nodes of $C$. If $C$ has no nodes, this follows from Proposition \ref{prop:analysis of eigenvalues}. For the general case, the same proposition implies that either $\operatorname{age}(\varphi) \geq 1$ or $\varphi$ lifts to a deformation with fewer nodes. Note that in the latter case the eigenvalues of $\varphi$ along the deformation vary continuously. Since these eigenvalues are roots of unity of order dividing $\operatorname{ord}(\varphi)$, we see that the eigenvalues are constant. Therefore the age is constant, and by the induction hypothesis applied to some fiber of the deformation, we get $\operatorname{age}(\varphi) \geq 1$.
\end{proof}

Let $J_0 \subseteq \overline{\mathcal{M}}_{3,n}$ be the locus of pointed curves that admit an elliptic tail of $j$-invariant $0$ without marked points. 
The following proposition is a local version of Proposition \ref{prop:condition to lift forms}, and implies Proposition \ref{prop:condition to lift forms}
 
\begin{proposition}\label{prop:lifting forms local version}
    Suppose that $n \geq 1$. Let $\eta$ be an $m$-canonical form defined on an open subset $U$ of $\Mbar{g}{n}$. If $U \cap J_0 \neq \emptyset$, we assume that the order of vanishing of $\eta$ along the divisor $\Delta_{1,\emptyset}$ is at least $m$.
    Then $\eta$ lifts to any resolution of singularities of $U$.
\end{proposition}
\begin{proof}
    Take a point $[C,p_1,...,p_n] \in U$.
    Let $V = \coho{0}{C,\Omega_C \otimes \omega_C(p_1,...,p_n)}$ and let $G = \operatorname{Aut}(C,p_1,...,p_n)$. Recall that a neighborhood of $[C,p_1,...,p_n] \in \overline{\mathcal{M}}_{3,n}$ is isomorphic to a neighborhood of the origin in $V^*/G$. 
    We will apply the different variations of the Reid-Shepherd-Barron-Tai criterion to deduce that $\eta$ lifts to a resolution of singularities of a neighborhood of the origin in $V^*/G$. This is enough to prove the proposition.
    We split in several cases, depending on wether the point lies in certain special loci of $\overline{\mathcal{M}}_{3,n}$. Let $X \subseteq \overline{\mathcal{M}}_{3,n}$ be the locus of curves admitting two distinct elliptic or nodal rational tails without marked points that are isomorphic. This locus is the image of a map
    \[
    \overline{\mathcal{M}}_{1,1} \times \overline{\mathcal{M}}_{1,n+2} \lra \overline{\mathcal{M}}_{3,n}
    \]
    and therefore has codimension at least $3$.
    Let $Y \subseteq \overline{\mathcal{M}}_{2,n+1}$ be the locus of pointed curves $(C',p_1,...,p_{n+1})$ that have an automorphism fixing the marked points of order at least $3$.
    Let $Z \subseteq \overline{\mathcal{M}}_{3,n}$ be the image of $Y \times \overline{\mathcal{M}}_{1,1}$ under the gluing map
    \[
    \overline{\mathcal{M}}_{2,n+1} \times \overline{\mathcal{M}}_{1,1} \lra \Delta_{1,\emptyset} \subset \overline{\mathcal{M}}_{3,n}.
    \]
    Note that $Y$ has codimension at least $2$ and therefore $Z$ has codimension at least $3$.
    
    \textbf{Case 1:} $[C,p_1,...,p_n] \not\in\Delta_{1,\emptyset}.$ In this case $G$ has no junior elements, by Proposition \ref{curves admitting a junior automorphism}, and hence $\eta$ lifts to a resolution of a neighborhood of the origin in $V^*/G$.

    \textbf{Case 2:} $[C,p_1,...,p_n] \in \Delta_{1,\emptyset}\backslash(X \cup Z)$.
    In this case
    \[
    (C,p_1,...,p_n)=(C',p_1,...,p_n,p)\cup_p (E,p)
    \]
    has an elliptic tail $E$ or rational tail $E$ with a node and all non-trivial automorphisms of $(C',p_1,...,p_n,p)$ have order $2$.
    Let $\varphi \in \operatorname{Aut}(C,p_1,...,p_n)$. Then $\varphi$ must map $E$ to $E$, since $(C,p_1,...,p_n)$ does not have two isomorphic elliptic or nodal rational tails. Therefore $\varphi$ sends $C'$ to $C'$, and $\restr{\varphi^2}{C'} = \operatorname{id}$.
    From the short exact sequence \eqref{eq:fundamental short exact sequence}, the list of eigenvalues of $\varphi$ on $V$ is one of the following, where $\zeta_k$ denotes a primitive $k$-root of unity:
    \[
    (\pm1,...,\pm1),(\pm i,-1,\pm1,...,\pm1) \text{ or } (\pm\zeta_6,\zeta_6^2,\pm1,...,\pm1).
    \]
    The order of $\restr{\varphi}{E}$ is $1$ or $2$, $4$, and $3$ or $6$ respectively.
    The list involving $\zeta_6$ can only occur when $E$ has $j$-invariant $0$, and in each case we have ordered the eigenvalues such that first eigenvalue corresponds to the action of $\varphi$ on a smoothing parameter $\tau_P$ for $P$. Note that $\Delta_{1,\varphi} \subseteq \Mbar{3}{n}$ contains the image of locus in $V^*$ where $\tau_P$ is zero.
    This means that in the notation of Lemma \ref{Reid-Tai with vanishing} we have $b_1 \geq m$ for the last list of eigenvalues.
    If the eigenvalues of $\varphi$ are all $\pm1$, then $V^*/\langle\varphi\rangle$ has canonical singularities. In the other cases, $\eta$ lifts to a resolution of singularities, by Examples \ref{ex:j=1728} and \ref{ex:j=0}. 
    This finishes the proof in case 2.\\

    \textbf{Case 3:} $[C,p_1,...,p_n]\in X\cup Z$. Let $W \subseteq U$ be the open subset consisting of points $x$ for which $\eta$ lifts to a pluri-canonical form on a resolution of singularities of a neighborhood of $x$. Let $\tilde{W}$ be the corresponding open subset of $V^*/G$. By the previous cases, $W$ contains $U\backslash (X \cup Z)$. In particular the codimension of the complement of $\tilde{W}$ in a small enough neighborhood of the origin of $V^*/G$ is at least $3$.
    Now suppose that $\varphi \in G$ is junior. Then the fixed locus of $\varphi$ in $V^*$ has codimension at most $2$, by Proposition \ref{prop:analysis of eigenvalues} and the fact that this codimension (which depends only on the eigenvalues) is constant along a deformation to which $\varphi$ lifts. In particular the fixed locus intersects $\tilde{W}$. The proof is finished by applying Proposition \ref{prop:lifting forms HM version}.
\end{proof}

As an interesting consequence, we can now easily determine the non-canonical locus of $\overline{\mathcal{M}}_{3,n}$.

\begin{corollary}\label{cor:non-canonical locus}
    Let $n \geq 1$. Then the locus of non-canonical singularities of $\Mbar{3}{n}$ equals $J_0$, the locus of pointed curves that admit an elliptic tail with $j$-invariant $0$ without marked points.
\end{corollary}
\begin{proof}
    By Proposition \ref{prop:lifting forms local version} we see that $\overline{\mathcal{M}}_{3,n}\backslash J_0$ has canonical singularities. Let
    \[
    (C,p_1,...,p_n)=(C',p_1,...,p_n,p)\cup_p (E,p)
    \]
    be a pointed curve with an elliptic tail $E$ of $j$-invariant $0$ and suppose that $C'$ is smooth and $(C',p_1,...,p_n,p)$ has no non-trivial automorphisms.
    Then $\operatorname{Aut}(C,p_1,...,p_n)$ is cyclic of order $6$ and a generator acts on $\coho{0}{\Omega_C \otimes \omega(p_1+...+p_n)}$ with eigenvalues $\zeta_6,\zeta_6^2,1,...,1$. It follows from Example \ref{ex:j=0} that $\coho{0}{\Omega_C \otimes \omega(p_1+...+p_n)}^*/\operatorname{Aut}(C,p_1,...,p_n)$ has a non-canonical singularities. This proves that the general point of $J_0$ is a non-canonical singularity. The same must then be true for all points of $J_0$, because the non-canonical locus is closed.
    \end{proof}

\section{Rigid component of the canonical divisor}\label{sec:rigid component}
In this section we prove Proposition \ref{prop:rigid component}.
Our strategy is based on the following standard way to identify rigid components, see for example \cite[Lemma 6.4]{BMS}

\begin{lemma}\label{lemma:identifying rigid components}
    Suppose $X$ is a normal $\qq$-factorial variety, $D$ an effective divisor on $X$ and $E \subseteq X$ an irreducible divisor. Suppose that there exists a curve class $\gamma \in \mathrm{N}_1(X)$ such that irreducible curves with class $\gamma$ cover a Zariski dense subset of $E$ and such that
\[
\begin{array}{c}
    \gamma\cdot D < 0\\
    \gamma\cdot E<0.
\end{array}
\]
Then $\frac{\gamma\cdot D}{\gamma\cdot E}E$ is a rigid component of $D$. In particular,
\[
\coho{0}{X,\mathcal{O}_X(D)} = \coho{0}{X,\mathcal{O}_X\left(D - \frac{\gamma\cdot D}{\gamma\cdot E}E\right)}
\]
\end{lemma}

We will use the curve class $\gamma \in \mathrm{N}_1\left( \Mbar{3}{n} \right)$ obtained by gluing a fixed stable pointed curve to a variable elliptic tail. More precisely, fix a stable pointed curve $[C',p_1,...,p_{n+1}] \in \Mbar{2}{n+1}$ and consider the morphism
\[
\theta:\Mbar{1}{1} \lra \Mbar{3}{n}
\]
that maps $[E,p]$ to $[C,p_1,...,p_n]$ where $C$ is obtained from gluing $p \in E$ to $p_{n+1} \in C'$. We define $\gamma \in \rm{N}_1\left( \Mbar{3}{n} \right)$ as the class of the image of $\theta$.

Let $\rho : \mathcal{X} \lra \Mbarstack{g}{n}$ be the universal curve and let $S_i$ be the image of the $i$-th marking.
The first Mumford class $\kappa_1$ is defined as\footnote{We follow the conventions of \cite[Chapter 13]{ACG}.}
\[
\kappa_1 = \rho_*\left( c_1\Big(\omega_\rho\big(\sum_{i=1}^n S_i\big)\Big)^2 \right).
\]
and equals
\begin{equation}\label{eq:Mumfords relation}
    \kappa_1 = 12\lambda + \psi_1 + ... + \psi_n -\delta,
\end{equation}
see \cite[Chapter 13, Theorem 7.6]{ACG}.
\begin{lemma}\label{lemma:intersection numbers elliptic tail curves}
    The intersection numbers of $\gamma$ with the divisors $\kappa_1,\psi_i,\delta_{irr}$ and $\delta_{i,S}$ are given as follows. We have
    \[
    \begin{array}{rl}
        \gamma \cdot \kappa_1 =& \frac{1}{12} \\
        \gamma \cdot \delta_{irr} =& 1 \\
        \gamma \cdot \delta_{1,\emptyset} =& -\frac{1}{12},
    \end{array}
    \]
    and all the other intersection numbers are $0$.
\end{lemma}
\begin{proof}
    We use the fact that $\gamma \cdot D = \deg(\theta^*D)$ for any $D \in \operatorname{Pic}_\mathbb{Q}\left(\Mbar{3}{n}\right)$.
    The formulas in \cite[Chapter 17, Lemma 4.38]{ACG} become
    \[
    \begin{array}{rl}
        \theta^*\kappa_1 =& \kappa_1 \\
        \theta^*\delta_{irr} =& \delta_{irr} \\
        \gamma \cdot \delta_{1,\emptyset} =& -\psi_1,
    \end{array}
    \]
    and the other elements of the basis are pulled back to zero. On $\Mbar{1}{1}$ there are the relations $12\kappa_1 = 12\psi_1 = \delta_{irr}$ and $\delta_{irr}$ is the divisor corresponding to a point. Therefore $\deg(\kappa_1) = \deg(\psi_1) = \frac{1}{12}$ and $\deg(\delta_{irr}) = 1$ and we are done.
\end{proof}

Note that by Lemma \ref{lemma:intersection numbers elliptic tail curves} the class $\gamma$ is independent of the choice of $[C',p_0,...,p_n]$ and therefore irreducible curves with class $\gamma$ cover the boundary divisor $\Delta_{1,\emptyset}$.

Equation \eqref{eq:Mumfords relation} allows us to rewrite the canonical divisor on $\Mbar{3}{n}$ as
\begin{equation}\label{eq:canonical divisor in kappa basis}
    K_{\Mbar{3}{n}} = \frac{13}{12}\kappa_1 - \frac{1}{12}(\psi_1+...+\psi_n) - \frac{11}{12}\delta - \delta_{1,\emptyset}.
\end{equation}

\begin{proof}[Proof of Proposition \ref{prop:rigid component}.]
    We apply Lemma \ref{lemma:identifying rigid components} to the curve class $\gamma$ and the divisors $mK_{\Mbar{3}{n}}$ and $\Delta_{1,\emptyset}$. The intersection numbers are easily computed with Lemma \ref{lemma:intersection numbers elliptic tail curves} and Equation \eqref{eq:canonical divisor in kappa basis}:
    \[
    \begin{array}{rl}
        \gamma \cdot mK_{\Mbar{3}{n}} =& -\frac{2m}{3} \\
        \gamma \cdot \Delta_{1,\emptyset} =& -\frac{1}{6}.
    \end{array}
    \]
    This proves the proposition.
\end{proof}

\section{Bigness of the canonical class of $\Mbar{3}{n}$}\label{sec:K is big}
In this section, we prove Theorem \ref{Main thm: K is big}. Concretely, when $n \geq 15$, we will express the canonical class as a positive linear combination of the Hodge class $\lambda$, the $\psi$-classes $\psi_1,...,\psi_n$, pull backs by forgetful maps of effective divisors considered in \cite{HM,Fa1}, and boundary divisors.
We will work with divisors on $\Mbar{g}{n}$ that are invariant under the action of the symmetric group $S_n$, and it will be convenient to define
\[
\psi = \psi_1 + ... + \psi_n,\,\,\,\, \delta_{i,k} = \sum_{\tiny{\begin{array}{c}
    S \subseteq \set{1,...,n} \\
    |S| = k
\end{array}}} \delta_{i,S}.
\]
Consider the forgetful morphism $\pi: \overline{\mathcal{M}}_{g,n+1} \lra \overline{\mathcal{M}}_{g,n}$ that forgets the last marked point.
The pull-back
\[
\pi^*:\operatorname{Pic}_\mathbb{Q}\left(\Mbar{g}{n}\right) \lra \operatorname{Pic}_\mathbb{Q}\left(\Mbar{g}{n+1}\right)
\]
is given by \cite[Chapter 17, Lemma 4.28]{ACG}:
\begin{equation}\label{eq:formulas pull back forgetful map}
    \begin{cases}
    \pi^*\lambda = \lambda,\\
        \pi^*\delta_{irr} = \delta_{irr},\\
        \pi^*\psi_j = \psi_j - \delta_{0,\set{j,n+1}} \text{ for } j=1,...,n,\\
        \pi^*\delta_{i,S} = \begin{cases}
            \delta_{i,S} \text{ if } n=0, i=\frac{g}{2}, S=\emptyset,\\
            \delta_{i,S} + \delta_{i,S \cup \{n+1\}} \text{ otherwise}.
        \end{cases}\\
\end{cases}
\end{equation}
It will be useful in computations to work with the classes $\omega_1,...,\omega_n,\omega \in \operatorname{Pic}_\mathbb{Q}\left(\Mbar{g}{n}\right)$ defined by
\[
\omega_j = \pi_j^*\psi_1 \text{ for } j=1,...,n \text{ and } \omega = \sum_{j=1}^n \omega_j,
\]
where $\pi_j:\Mbar{g}{n} \lra \Mbar{g}{1}$ forgets all points except the $j$-th marked point.
Observe that, by induction on $n$, we have that
\[
\omega_j = \psi_j - \sum_{S \ni j} \delta_{0,S} \text{ for } j=1,...,n\,\,\, \text{ and }\,\,\, \omega = \psi - \sum_{k=2}^n k\delta_{0,k},
\]
where the first sum runs over all subsets $S \subseteq \{1,...,n\}$ containing $j$ and at least one other element.\\
Let $\overline{H}$ be the divisor in $\mbar_3$ defined as the closure of the locus of hyperelliptic curves $H \subseteq \mathcal{M}_3$. The class of $\overline{H}$ in $\operatorname{Pic}_\mathbb{Q}\left( \mbar_3 \right)$ was computed in \cite[Section 6]{HM}:
\[
\overline{H} = 9\lambda - \delta_{irr} -3\delta_{1}.
\]
Let $g,r \geq 1$ and set $n=(2r+1)(g+1)$.
Let $\overline{D}_{g,n}^r$ be the divisor in $\Mbar{g}{n}$ defined as the closure in $\Mbar{g}{n}$ of the locus
\[
D_{g,n}^r = \setcondition{[C,p_1,...,p_n] \in \M{g}{n}}{\operatorname{h}^0\left(C,\omega_C^{\otimes(r+1)}(-p_1-...-p_n)\right) \geq 1} \subseteq \M{g}{n}.
\]
The divisor $\overline{D}_{g,n}^r$ was introduced by Farkas, who denoted it by $\mathfrak{Mrc}_{g,0}^r$, and computed the class \cite[Theorem 4.2]{Fa1}:
\begin{equation}\label{eq:Farkas divisor}
    \overline{D}_{g,n}^r = -(6r^2+6r+1)\lambda + (r+1)\omega + \binom{r+1}{2} \cdot \delta_{irr} -  \delta_{0,2} - \BD.
\end{equation}
Finally, recall from \cite[Theorem 2.9]{Lo} that the class
\[
a\lambda + b\psi \in \operatorname{Pic}_\mathbb{Q}\left(\Mbar{g}{n}\right)
\]
is big for $a,b > 0$.

\begin{proof}[\textit{Proof of Theorem \ref{Main thm: K is big}}]
Let $n \geq 15$. We consider two effective divisors on $\overline{\mathcal{M}}_{3,n}$.
The first one is the pull-back of $\overline{H}$ under the forgetful map $\pi:\Mbar{3}{n} \lra \mbar_3$:
\begin{equation}
\pi^*\overline{H} = 9\lambda - \delta_{irr} - 3\delta_1.
\end{equation}
The second one is the \textit{symmetric pull-back} of $\overline{D}_{3,14}^3$ to $\Mbar{3}{n}$, defined as
\[
\overline{D}_n = \frac{1}{\binom{n}{14}}\sum_{\tiny\begin{array}{c}
    T \subseteq \set{1,...,n} \\
    |T| = 14
\end{array}} \pi_T^* \overline{D}_{3,14}^3,
\]
where $\pi_T:\Mbar{3}{n} \lra \Mbar{3}{14}$ is the forgetful map that forgets all points except for those indexed by $T$.
Equation \eqref{eq:Farkas divisor} becomes
\[
\overline{D}_{3,14}^3 = -73\lambda + 4\omega + 6\delta_{irr} - \delta_{0,2} - \BD,
\]
and using the formulas for the pull back \eqref{eq:formulas pull back forgetful map} we obtain that
\begin{equation}
    \begin{array}{rl}
    \overline{D}_n = & -73\lambda + \frac{56}{n}\omega + 6\delta_{irr} - \frac{\binom{14}{2}}{\binom{n}{2}}\delta_{0,2} - \BD \\
    = & -73\lambda + \frac{56}{n}\left(\psi - \sum_{k=2}^n k\delta_{0,k} \right) + 6\delta_{irr} - \frac{182}{n(n-1)} \delta_{0,2} - \BD\\
    = & -73\lambda + \frac{56}{n}\psi + 6\delta_{irr} - \frac{182}{n(n-1)}\delta_{0,2} - \sum_{k=2}^n \frac{56k}{n}\delta_{0,k} - \BD.
\end{array}
\end{equation}
By equation \eqref{eq:formula canonnical class}, the canonical class on $\Mbar{3}{n}$ equals
\[
\begin{array}{rl}
    K_{\overline{\mathcal{M}}_{3,n}}= & 13\lambda + \psi - 2\delta_{irr} - 2 \delta_0 - 2\delta_1 -  \delta_{1,\emptyset} \\
    = & 13\lambda + \psi - 2\delta_{irr} - 2 \delta_0 - 3\delta_1 + \BD.
\end{array}
\]
Set $t = \frac{n}{56}$ and $s = \frac{13+73t}{9}$. These numbers have been chosen so that the coefficients of $\lambda$ and $\psi$ in $s\pi^*\overline{H} + t\overline{D}_n$ and $K_{\Mbar{g}{n}}$ are equal.
Then $K_{\overline{\mathcal{M}}_{3,n}} - s\pi^*\overline{H}- t\overline{D}_n$ equals
\begin{equation}\label{eq:canonical - effective}
    \left(s - 6t -2\right)\delta_{irr} + (3s-3)\delta_{1} + \frac{182t}{n(n-1)}\delta_{0,2} + \sum_{k=3}^n (k-2)\delta_{0,k} + \BD.
\end{equation}
Since $n \geq 15$, the coefficients in front of $\delta_{irr},\delta_1$ and $\delta_{0,k}$ in \eqref{eq:canonical - effective} are positive. It follows that for $\epsilon > 0$ small enough
\[
 K_{\overline{\mathcal{M}}_{3,n}} - (s-9\epsilon)\pi^*\overline{H} - (t-\epsilon)\overline{D}_n = 8\epsilon\lambda + \frac{56\epsilon}{n} \psi + \BD.
\]
Thus $K_{\Mbar{3}{n}}$ is a sum of big and effective classes, which proves that $K_{\Mbar{3}{n}}$ is big.\\
\end{proof}

\end{document}